\documentclass[12pt,letterpaper,reqno]{amsart}

\usepackage{amssymb,amsmath}
\usepackage{amsthm}
\usepackage{color,graphicx}
\usepackage{eucal}

\newtheorem{theorem}{Theorem}[section]
\newtheorem{proposition}[theorem]{Proposition}
\newtheorem{remark}[theorem]{Remark}
\newtheorem{lemma}[theorem]{Lemma}
\newtheorem{corollary}[theorem]{Corollary}
\newtheorem{definition}[theorem]{Definition}

\newcommand{\R}{\mathbb R}

\newcommand{\la}{{\lambda}}

\newcommand{\ep}{\varepsilon}

\newcommand{\ds}{\displaystyle}

\newcommand{\al}{\alpha}

\newcommand{\cR}{\mathbb R}

\numberwithin{equation}{section}

\begin{document}


\title[Instablity of solitons in gZK]
{Instability of solitons - revisited, II:\\
the supercritical Zakharov-Kuznetsov equation}

\author[L. G. Farah]{Luiz Gustavo Farah}
\address{Department of Mathematics\\UFMG\\Brazil}
\curraddr{}
\email{lgfarah@gmail.com}
\thanks{}

\author[J. Holmer]{Justin Holmer}
\address{Department of Mathematics\\Brown University\\ USA}
\email{holmer@math.brown.edu}
\thanks{}

\author[S. Roudenko]{Svetlana Roudenko}
\address{Department of Mathematics\\George Washington University\\ USA}
\curraddr{}
\email{roudenko@gwu.edu}
\thanks{}



\begin{abstract}
We revisit the phenomenon of instability of solitons in the two dimensional generalization of the Korteweg-de Vries equation, the generalized Zakharov-Kuznetsov (ZK) equation, $u_t + \partial_{x_1}(\Delta u + u^p) = 0,  (x_1,x_2) \in \R^2$. It is known that solitons are unstable in this two dimensional equation for nonlinearities $p > 3$. This was shown by Anne de Bouard in \cite{DeB} generalizing the arguments of Bona-Souganidis-Strauss in \cite{BSS} for the generalized KdV equation. In this paper, we use a different method to obtain the instability of solitons, namely, truncation and monotonicity properties. Not only does this approach simplify the proof, but it can also be useful for studying various other stability questions in the ZK equation as well as other generalizations of the KdV equation.
\end{abstract}

\maketitle

\tableofcontents

\section{Introduction}

In this paper we consider the two-dimensional generalized Zakharov-Kuznetsov equation:
\begin{equation}
\label{gZK}
\text{(gZK)} \qquad \qquad u_t + \partial_{x_1} \left( \Delta_{(x_1,x_2)} u + u^p \right) = 0, \quad p>3, \quad (x_1,x_2) \in \cR^2, \quad t \in \cR.
\end{equation}
The above equation with $p=2$, besides being the 2d extension of the well-known KdV equation, governs the behavior of weakly nonlinear ion-acoustic waves in plasma comprising of cold ions and hot isothermal electrons in the present of a uniform magnetic field \cite{MP-1, MP-2} and was originally derived by Zakharov and Kuznetsov to describe weakly magnetized ion-acoustic waves in a strongly magnetized plasma \cite{ZK}.

The equation \eqref{gZK} is the two-dimensional extension of the well-studied model describing, for example, the weakly nonlinear waves in shallow water, the Korteweg-de Vries (KdV) equation:
\begin{equation}
\label{gKdV}
\text{(KdV)} \qquad \qquad
u_t + \left( u_{xx} + u^p \right)_x = 0, \quad p=2, \quad x \in \cR, \quad t \in \cR. \qquad
\end{equation}
When other integer powers $p \neq 2$ are considered, it is referred to as the {\it generalized} KdV (gKdV) equation.
It is also possible to consider
non-integer powers $p>1$, however, the nonlinearity would have to be modified as $\partial_{x_1} (|u|^{p-1}u )$. For the odd powers, this would produce a slightly different equation, however, most of the theory would remain the same.
Despite its apparent universality, the gKdV equation is limited as a spatially one-dimensional model, and thus, various higher dimensional generalizations exist.

During their lifespan, the solutions $u(t, x_1,x_2)$ to the equation \eqref{gZK} conserve the mass and energy:
\begin{equation}\label{MC}
M[u(t)]=\int_{\cR^2} u^2(t)\, dx_1dx_2 = M[u(0)]
\end{equation}
and
\begin{equation}\label{EC}
E[u(t)]=\dfrac{1}{2}\int_{\cR^2}|\nabla u(t)|^2\;dx_1dx_2 - \dfrac{1}{4}\int_{\cR^2} u^{4}(t)\;dx_1dx_2 = E[u(0)].
\end{equation}

Similar to the gKdV equation,
for solutions $u(t,x,y)$ decaying at infinity on $\cR^2$ the following invariance holds
\begin{equation*}
\int_{\cR} u(t, x_1,x_2) \, dx_1 = \int_{\cR} u(0, x_1, x_2) \, dx_1,
\end{equation*}
which is obtained by integrating the original equation on $\cR$ in the first coordinate $x_1$.
\smallskip

One of the useful symmetries in the evolution equations is {\it the scaling invariance}, which states that an appropriately rescaled version of the original solution is also a solution of the equation. For the equation \eqref{gZK} it is
$$
u_\lambda(t, x_1,x_2)=\lambda^{\frac{2}{p-1}} u(\lambda^3t, \lambda x_1, \lambda x_2).
$$
This symmetry makes a specific Sobolev norm $\dot{H}^s$ invariant, i.e.,
\begin{equation*}
\|u(0,\cdot,\cdot) \|_{\dot{H}^s}=\lambda^{\frac{2}{p-1}+s-1} \|u_0\|_{\dot{H}^s},
\end{equation*}
and the index $s$ gives rise to the critical-type classification of equations.
For the gKdV equation \eqref{gKdV} the critical index is $s=\frac12-\frac2{p-1}$, and for the gZK equation \eqref{gZK} it is $s = 1- \frac2{p-1}$. When $s>0$ (in 2d gZK equation this corresponds to $p>3$), the equation \eqref{gZK} is often referred as the ($L^2$-) supercritical equation. The gZK equation has other invariances such as translation and dilation.

The gZK equation has a family of travelling waves (or solitary waves, which sometimes are referred even as solitons), and observe that they travel only in $x_1$ direction
\begin{equation}\label{Eq:TW}
u(t,x_1,x_2) = Q_c(x_1-ct, x_2)
\end{equation}
with $Q_c(x_1,x_2) \to 0$ as $|x| \to + \infty$. Here, $Q_c$ is the dilation of the ground state $Q$:
$$
Q_c(\vec{x}) = c^{1/{p-1}} Q (c^{1/2} \vec{x}), \quad \vec{x} = (x_1,x_2),
$$
with $Q$ being a radial positive solution in $H^1(\cR^2)$ of the well-known nonlinear elliptic equation $-\Delta Q+Q  - Q^p = 0$.
Note that $Q \in C^{\infty}(\R^2)$, $\partial_r Q(r) <0$ for any $r = |x|>0$ and
for any multi-index $\alpha$
\begin{equation}\label{prop-Q}
|\partial^\al Q(\vec{x})| \leq c(\al) e^{- |\vec{x}|} \quad \mbox{for any}\quad \vec{x} \in \cR^2.
\end{equation}

In this work, we are interested in stability properties of travelling waves in
the supercritical gZK equation \eqref{gZK}, i.e., in the behavior of solutions close to the ground state $Q$ (perhaps, up to translations). We begin with the precise concept of stability and instability used in this paper. For $\alpha>0$, the neighborhood (or ``tube") of radius $\alpha$ around $Q$ (modulo translations) is defined by
\begin{equation*}
U_{\alpha}=\left\{u\in H^1(\mathbb{R}^2): \inf_{\vec{y}\in \mathbb{R}^2}\|u(\cdot)-Q(\cdot+\vec{y})\|_{H^1}\leq \alpha \right\}.
\end{equation*}
\begin{definition}[Stability of $Q$]
We say that $Q$ is stable if for all $\alpha>0$, there exists $\delta>0$ such that if $u_0\in U_{\delta}$, then the corresponding solution $u(t)$ is defined for all $t\geq 0$ and $u(t)\in U_{\alpha}$ for all $t\geq 0$.
\end{definition}

\begin{definition}[Instability of $Q$]
We say that $Q$ is unstable if $Q$ is not stable, in other words, there exists $\alpha>0$ such that for all $\delta>0$ the following holds:
if $u_0\in U_{\delta}$, then there exists $t_0=t_0(u_0)$ such that $u(t_0)\notin U_{\alpha}$.
\end{definition}

The main goal in this paper is to provide a new proof of the instability result, originally obtained by de Bouard \cite{DeB} in her study of dispersive solitary waves in higher dimensions (her result holds in dimensions 2 and 3). She showed \cite{DeB} that the travelling waves of the form \eqref{Eq:TW} are stable (in the two dimensional case) for $p<3$ and unstable for $p>3$. She followed the ideas of Bona-Souganidis-Strauss \cite{BSS} for the instability, and Grillakis-Shatah-Strauss \cite{GSS} for the stability arguments. Here, we prove the instability of the traveling wave solution of the form \eqref{Eq:TW}, in a spirit of Combet \cite{Co10}, where he revisited the instability phenomenon for the gKdV equation in the supercritical case, $p>5$. We show an open set of initial data, in particular, an explicit example of a sequence of initial data, which would contradict the stability of $Q$. To this end consider, for $n\geq 1$
$$
u_{0,n}(\vec{x})=\lambda_nQ(\lambda_n \vec{x}), \quad \textrm{where} \quad \lambda_n=1+\frac1n \quad \textrm{and} \quad  \vec{x}=(x_1,x_2).
$$
Our main motivation is to show the new methods available (monotonicity and truncation) to obtain the following result:
\begin{theorem}[$H^1$-instability of $Q$ for the supercritical gZK]\label{Theo-Inst}
Let $u_n$ be the solution with initial data $u_{0,n}$, then there exists $\alpha>0$ such that for every $n\in \mathbb{N}$,
there exists $T_n=T_n(u_{0,n})$ such that $u_n(T_n)\notin U_{\alpha}$, or explicitly,
$$
\inf_{\vec{y}\in \R^2}\|u_n(T_{n}, \cdot)-Q(\cdot-\vec{y})\|\geq \alpha.
$$
\end{theorem}

The paper is organized as follows. In Section \ref{S-2} we provide the background information on the well-posedness of the generalized ZK equation in 2 dimensions. In Section \ref{S-3} we discuss the properties of the linearized operator $L$ around the ground state $Q$. Section \ref{S-4} contains the canonical decomposition of the solution around $Q$, the modulation theory and control of parameters coming from such a decomposition. In Section \ref{S-5} we discuss the virial-type functional and the concept of monotonicity. The next Section \ref{S-6} contains a new proof of the instability via truncation and monotonicity.
\medskip

{\bf Acknowledgements.} Most of this work was done when the first author was visiting GWU in 2016-17 under the support of the Brazilian National Council for Scientific and Technological Development (CNPq/Brazil), for which all authors are very grateful as it boosted the energy into the research project.
S.R. would like to thank IHES and the organizers for the excellent working conditions during the trimester program ``Nonlinear dispersive waves - 2017" in May-July 2016.
L.G.F. was partially supported by CNPq and FAPEMIG/Brazil.
J.H. was partially supported by the NSF grant DMS-1500106.
S.R. was partially supported by the NSF CAREER grant DMS-1151618.

\section{Background on the generalized ZK equation} \label{S-2}
In this section we review the known results on the local and global well-posedness of the generalized ZK equation. To follow the notation
in the literature, in this section we denote the power of nonlinearity as $u^{k+1}$ (instead of $u^p$) and consider the Cauchy problem for the gZK equation as follows:
\begin{equation}\label{gzk2}
\begin{cases}
u_t+\partial_{x_1} \Delta u+ \partial_{x_1}(u^{k+1})  =  0,  \quad (x_1,x_2) \in \mathbb{R}^2, ~ t>0, \\
u(0, x_1,x_2)=u_0(x_1,x_2) \in {H}^s(\R^2).
\end{cases}
\end{equation}
The first paper to address the local well-posedness of this Cauchy problem for the $k=1$ case was by Faminskii \cite{Fa}, where he considered $s=1$ (strictly speaking, he obtained the local well-posedness in $H^m$, for any integer $m \geq 1$.)
The current results on the local well-posedness are gathered in the following statement.
\begin{theorem} 
The local well-posedness in \eqref{gzk2} holds in the following cases:
\begin{itemize}
\item
$k=1$: for $s>\frac12$, see Gr\"unrock-Herr \cite{GH},
\item
$k=2$: for $s>\frac14$, see Ribaud-Vento \cite{RV},
\item
$k=3$: for $s>\frac5{12}$, see Ribaud-Vento \cite{RV},
\item
$k=4, 5, 6, 7$: for $s>1-\frac2{k}$, see Ribaud-Vento \cite{RV},
\item
$k=8$, $s>\frac34$, see Linares-Pastor \cite{LP1} 
\item
$k>8$, $s>s_k=1-2/k$, see Farah-Linares-Pastor \cite{FLP}.
\end{itemize}
\end{theorem}
Note that in the last three cases (i.e., for $k \geq 4$), the bound on $s > s_k$ is optimal from the scaling conjecture.
For previous results on the local well-posedness for $2 \leq k \leq 8$ for $s>3/4$ see \cite{LP} and \cite{LP1}.

Following the approach of Holmer-Roudenko for the $L^2$-supercritical nonlinear Schr\"odinger (NLS) equation, see \cite{HR} and \cite{DHR}, the first author together with F. Linares and A. Pastor obtained the global well-posedness result for the nonlinearities $k \geq 3$ and under a certain mass-energy threshold, see \cite{FLP}.
\begin{theorem}[\cite{FLP}]
 Let $k\geq3$ and $s_k=1-2/k$.  Assume $u_0\in H^1(\R^2)$ and suppose that
\begin{equation}\label{GR1}
E(u_0)^{s_k} M(u_0)^{1-s_k} < E(Q)^{s_k} M(Q)^{1-s_k} , \,\,\,
E(u_0) \geq 0.
\end{equation}
If
\begin{equation}\label{GR2}
\|\nabla u_0\|_{L^2}^{s_k}\|u_0\|_{L^2}^{1-s_k} < \|\nabla
Q\|_{L^2}^{s_k}\|Q\|_{L^2}^{1-s_k},
\end{equation}
then for any $t$ from the maximal interval of existence 
\begin{equation*}
\|\nabla u(t)\|_{L^2}^{s_k}\|u_0\|_{L^2}^{1-s_k}=\|\nabla
u(t)\|_{L^2}^{s_k}\|u(t)\|_{L^2}^{1-s_k} <\|\nabla
Q\|_{L^2}^{s_k}\|Q\|_{L^2}^{1-s_k},
\end{equation*}
where $Q$ is the unique positive radial solution of
$$
\Delta Q-Q+Q^{k+1}=0.
$$
In particular, this implies that $H^1$ solutions, satisfying \eqref{GR1}-\eqref{GR2} exist globally in time.
\end{theorem}

\begin{remark}
In the limit case $k=2$ (or $p=3$, the  modified ZK equation), conditions
\eqref{GR1} and \eqref{GR2} reduce to one condition, which is
$$
\|u_0\|_{L^2}<\|Q\|_{L^2}.
$$
Such a condition was already used in \cite{LP} and \cite{LP1} to
show the existence of global solutions, respectively, in $H^1(\R^2)$
and $H^s(\R^2)$, $s>53/63$.
\end{remark}
We conclude this section with a note that for the purpose of this paper, it is sufficient to have the well-posedness theory in $H^1(\R^2)$.

\section{The Linearized Operator $L$}\label{S-3}

The operator $L$, which is obtained by linearizing around the ground state $Q$, is defined by
\begin{equation}
\label{L-def}
L  := - \Delta  + 1 - p \, Q^{p-1}.
\end{equation}
We first state the properties of this operator $L$ (see Kwong \cite{K89} for all dimensions, Weinstein \cite{W85} for dimension 1 and 3, also Maris \cite{M02} and \cite{CGNT}).
\begin{theorem}[Properties of $L$]\label{L-prop}
The following holds for an operator $L$ defined in \eqref{L-def}
\begin{itemize}
\item
$L$ is a self-adjoint operator and
$\sigma_{ess}(L) = [ \la_{ess}, +\infty ) \quad \mbox{for~~ some~~} \la_{ess} > 0$

\item
$\ker L = \mbox{span} \{Q_{x_1}, Q_{x_2} \}$
	
\item
$L$ has a unique single negative eigenvalue $-\la_0$ (with $\la_0 > 0$) associated to a
positive radially symmetric eigenfunction $\chi_0$.
Moreover, there exists $\delta > 0$ such that
\begin{equation}\label{chi_0}
|\chi_0(x)| \lesssim e^{- \delta |x|} \quad \mbox{for ~~all} ~~ x \in \cR^2.
\end{equation}
\end{itemize}
\end{theorem}
We also define the generator $\Lambda$ of the scaling symmetry
\begin{equation*}
\Lambda f = \frac1{p-1} f + \frac12 \, \vec{x} \cdot \nabla f, ~~~(x_1,x_2) \in \cR^2.
\end{equation*}
The following identities are useful to have
\begin{lemma}\label{L-prop2}
The following identities hold
\begin{enumerate}
\item
$L (\Lambda Q) = - Q$ 
\item
$\int Q \, \Lambda Q = \frac{3-p}{2(p-1)}\int Q^2<0$ for $p>3$, and $\int Q \, \Lambda Q = 0$ if $p=3$.
\end{enumerate}
\end{lemma}

\begin{proof}
The first two identities follow directly from the definition of $L$, $\Lambda$ and the equation $-\Delta Q+Q  - Q^p = 0$. The third identity is also easy to derive, for example
\begin{equation*}
\int Q\Lambda Q =\frac{1}{p-1}\int Q^2-\frac12\left(2\int Q^2 + \int x \cdot \nabla Q \right)
= \left(\frac{1}{p-1}-\frac12\right)\int Q^2,
\end{equation*}
from which (iii) follows.

\end{proof}


%
%
%

In general, the operator $L$ is not positive-definite, however, if we exclude appropriately the zero eigenvalue and negative eigenvalue directions, then one can expect some positivity properties, which we exhibit in the following two lemmas.
\begin{lemma}[Orthogonality Conditions I]\label{Lemma-ort1}
Let $\chi_0$ be the positive radially symmetric eigenfunction associated to the unique single negative eigenvalue $-\la_0$ (with $\la_0 > 0$). Then, there exists $\sigma_0$ such that for any $f \in H^1(\R^2)$ satisfying
\begin{equation*}
(f, \chi_0) = (f, Q_{x_j}) =0, \quad j=1,2,
\end{equation*}
one has
$$
(Lf,f) \geq \sigma_0 \, \|f\|_2^2.
$$
\end{lemma}
\begin{proof}
The result follows directly from Schechter \cite[Chapter 8, Lemma 7.10]{Sch} (see also \cite[Chapter 1, Lemma 7.17]{Sch})
\end{proof}
\begin{lemma}[Orthogonality Conditions II]\label{Lemma-ort2}
There exist $k_1, k_2>0$ such that for all $\ep\in H^1(\R^2)$ satisfying $\ep \perp \{Q_{x_1}, Q_{x_2}\}$ one has
$$
(L\ep,\ep) = \int |\nabla \ep|^2+\int \ep^2 - p\int Q^{p-1}\ep \geq k_1 \, \|\ep\|_2^2-k_2|(\ep, \chi_0)|^2.
$$
\end{lemma}
\begin{proof}
If $\ep \perp \chi_0$, the result is true by the previous lemma. Now, if $(\ep, \chi_0)\neq 0$, let $a\in \R$ be such that $\ep_1=\ep-a\chi_0$ verifies $\ep_1\perp \chi_0$, in other words,
\begin{equation}\label{def-a}
(\ep_1, \chi_0)= (\ep, \chi_0)-a\|\chi_0\|^2_2=0 \Longleftrightarrow a=(\ep, \chi_0)\|\chi_0\|^{-2}_2.
\end{equation}
Therefore, by Lemma \ref{Lemma-ort1} we have
\begin{align}\label{ep_11}
(L\ep_1, \ep_1)\geq \sigma_0 \|\ep_1\|^2_2=\sigma_0(\|\ep\|_2^2-2a(\ep,\chi_0)+a^2\|\chi_0\|_2^2).
\end{align}
On the other hand,
\begin{align}\label{ep_12}
(L\ep_1, \ep_1)=&(L\ep, \ep)-2a(\ep,L\chi_0)+a^2(L\chi_0, \chi_0) \nonumber\\
=&(L\ep, \ep)-2a\lambda_0(\ep,\chi_0)+a^2\lambda_0\|\chi_0\|_2^2.
\end{align}
Collecting \eqref{ep_11} and \eqref{ep_12}, we get
\begin{align*}
(L\ep, \ep)\geq&\sigma_0\|\ep\|^2_2-2a(\sigma_0+\lambda_0)(\ep,\chi_0)+a^2(\sigma_0+\lambda_0)\|\chi_0\|_2^2 \\
=&\sigma_0\|\ep\|^2_2-(\sigma_0+\lambda_0)\|\chi_0\|_2^{-2}|(\ep,\chi_0)|^2,
\end{align*}
where we have used \eqref{def-a} in the last line. The result follows by taking $k_1=\sigma_0$ and $k_2=(\sigma_0+\lambda_0)\|\chi_0\|_2^{-2}$.
\end{proof}
Following \cite{Co10}, we introduce the Lyapunov-type functional, connecting with the linearized operator $L$ and obtain the upper bounds on it.
\begin{lemma}[Weinstein's Functional]\label{Weinstein-Functional}
Recall \eqref{MC}-\eqref{EC} and define
$$
W[u]=E[u]+\frac12 M[u].
$$
Then
\begin{equation}\label{FLH}
W[Q+\ep]=W[Q]+\frac12 (L\ep, \ep)+ K[\ep]
\end{equation}
with $K: H^1 \to \cR$, and if $\|\ep\|_{H^1}\leq 1$, there exists $C>0$ such that
\begin{equation}\label{H}
|H[\ep]|\leq C \, \|\ep\|_{H^1}\|\ep\|^2_{2}.
\end{equation}
\end{lemma}
\begin{proof}
A straighforward calculation reveals
\begin{align} \label{E:energy}
E[Q+\ep]  &= \frac{1}{2}\int |\nabla(Q+\ep)|^2-\frac{1}{p+1}\int (Q+\ep)^{p+1} \nonumber\\
&= E[Q] +\frac12\int |\nabla \ep|^2+\int \nabla Q\cdot \nabla \ep-\frac{1}{p+1}\int \left((p+1)Q^p\ep+\frac{(p+1)p}{2}Q^{p-1}\ep^2+G[\ep]\right) \nonumber\\
& =E[Q] +\frac12\int |\nabla \ep|^2-\int (\Delta Q+Q^p)\ep-\frac{p}{2}\int Q^{p-1}\ep^2-\frac{1}{p+1}G[\ep] \nonumber\\
& =E[Q] +\frac12\int |\nabla \ep|^2-\int Q\ep-\frac{p}{2}\int Q^{p-1}\ep^2+H[\ep],
\end{align}
where $H[\ep]=-\frac{1}{p+1}\int G[\ep]$ with\footnote{Recall that $\binom{p+1}{k}=\frac{(p+1)!}{(p+1-k)!k!}$.} $\displaystyle G[\ep]=\sum^{p+1}_{k=3}\binom{p+1}{k}Q^{p+1-k}\ep^k$ and we have used $Q=\Delta Q+ Q^p$ in the last line.

Since $Q\in L^{\infty}(\R^2)$, we use the Gagliardo-Nirenberg inequality
\begin{equation}\label{GN}
\|f\|_q^q\leq c \|\nabla f\|_2^{q-2}\|f\|_2^2, \quad q\geq 2,
\end{equation}
to deduce that, if $\|\ep\|_{H^1}\leq 1$, then
$$
|H[\ep]|\leq c \sum_{k=3}^{p+1}\|\ep\|_k^k \leq C \|\ep\|_{H^1}\|\ep\|_2^2,
$$
for some constant $C >0$. Next, noticing that
\begin{equation}\label{E:mass}
M[Q+\ep]=\int Q^2+2\int Q\ep +\int \ep^2,
\end{equation}
and putting together \eqref{E:mass} and \eqref{E:energy}, we obtain \eqref{FLH}.
\end{proof}

\section{Decomposition of $u$ and Modulation Theory}\label{S-4}

We consider the canonical parametrization of the solution $u(t, x_1, x_2)$ close to $Q$:
\begin{equation}
\label{ep-def}
\ep(t,x_1,x_2) = u(t, x_1+y_1(t), x_2+y_2(t)) - Q(x_1,x_2),
\end{equation}
where $u$ is a solution of \eqref{gZK} and $y_1(t)$, $y_2(t)$ are two $C^1$ functions to be determined later. In the next lemma we deduce the equation for $\ep(t,\vec{x})$.
\begin{lemma}[Equation for $\ep$]
There exists $C_0>0$ such that for all $t \geq 0$, we have
\begin{align}\label{ep1-order1}
\ep_t - (L \ep)_{x_1} =(y'_1(t)-1)(Q+\ep)_{x_1} + y'_2(t) (Q +\ep)_{x_2} - R(\ep),
\end{align}
where
\begin{align}\label{error}
|R(\ep)|\leq C_0\sum_{k=2}^p(|\ep|^k+|\ep_{x_1}||\ep|^{k-1}).
\end{align}
\end{lemma}
\begin{proof}
By definition \eqref{ep-def}, we have
$$
u(t,x_1,x_2)=Q(x_1-y_1(t), x_2-y_2(t))+\ep(t, x_1-y_1(t), x_2-y_2(t)).
$$
Since $u$ is a solution of \eqref{gZK}, we have
\begin{equation}\label{ep-eq2}
y_1'(t)(Q+\ep)_{x_1}+y_2'(t)(Q+\ep)_{x_2}-\ep_t=\partial_{x_1}\Delta (Q+\ep)+(Q^p)_{x_1}+p(Q^{p-1}\ep)_{x_1}+R(\ep),
\end{equation}
where
\begin{equation}\label{R(ep)}
R(\ep)=\partial_{x_1}\left(\sum_{k=2}^p \binom{p}{k}Q^{p-k}\ep^k\right)=\sum_{k=2}^p \binom{p}{k}((p-k)Q_{x_1}Q^{p-k-1}\ep^k+kQ^{p-k}\ep_{x_1}\ep^{k-1}).
\end{equation}
Recalling that $Q, Q_{x_1}\in L^{\infty}(\R^2)$, there exists a universal constant $C_0>0$ such that the estimate \eqref{error} holds.

Using that $-\Delta Q+Q  - Q^p = 0$ and the definition of $L$ from \eqref{L-def}, we rewrite the equation \eqref{ep-eq2} as
$$
-\ep_t+(L\ep)_{x_1}-\ep_{x_1} =Q_{x_1}-y_1'(t)(Q+\ep)_{x_1}-y_2'(t)(Q+\ep)_{x_2}+R(\ep),
$$
which implies the equation \eqref{ep1-order1}, concluding the proof.
\end{proof}

Next we recall the modulation theory close to the ground state $Q$, as in de Bouard \cite{DeB} (see also Bona-Souganidis-Strauss \cite{BSS} for a similar result in the gKdV model).
\begin{proposition}[Modulation Theory]\label{ModThI}
There exists $\alpha_1 >0$, $C_1>0$ and a unique $C^1$ map
$$
(y_1,y_2): U_{{\alpha_1}}\rightarrow \mathbb{R}^2
$$
such that if $u\in U_{{\alpha_1}}$ and $\ep_{({y_1}(u), {y_2}(u))}$ is given by
\begin{equation}\label{ep-def2}
\ep_{({y_1}(u), {y_2}(u))}(x_1,x_2) = u( x_1 +y_1(u), x_2 + y_2(u)) - Q(x_1,x_2),
\end{equation}
then
$$
\ep_{({y_1}(u), {y_2}(u))} \perp Q_{x_j}, \quad j=1,2.
$$
Moreover, if $u\in U_{{\alpha}}$ with $0<\alpha<\alpha_1$, then
\begin{equation}\label{ep-control}
\|\ep_{({y_1}(u), {y_2}(u))}\|_{H^1}\leq C_1\alpha.
\end{equation}

Furthermore, if $u$ is cylindrically symmetric (i.e., $u(x_1,x_2)=u(x_1, |x_2|)$), then, reducing $\alpha_0$ if necessary, we can assume $y_2(u)=0$.
\end{proposition}
\begin{proof}
See \cite[Lemma 4.4]{DeB}.
\end{proof}

Now, assume that $u(t)\in U_{\alpha}$, with $\alpha<\alpha_1$, for all $t\geq 0$, and define the functions $y_1(t)$ and $y_2(t)$ as follows.
\begin{definition}\label{eps}
For all $t\geq 0$, let $y_1(t)$ and $y_2(t)$ be such that $\ep_{({y_1}(t), y_2(t))}$, defined according to the equation \eqref{ep-def2}, satisfy
\begin{equation}\label{ep-perp2}
\ep_{({y_1}(t), y_2(t))} \perp Q_{x_j}, \quad j=1,2.
\end{equation}
In this case, we set
\begin{equation}\label{eq-ep2}
\ep(t)=\ep_{({y_1}(t), y_2(t))}=u( x_1 +y_1(t), x_2+y_2(t)) - Q(x_1,x_2).
\end{equation}
\end{definition}
From the estimate \eqref{ep-control}, it is clear that
\begin{equation}\label{ep-control2}
\|\ep(t)\|_{H^1}\leq C_1\alpha.
\end{equation}
The next lemma provides us with the estimates for $|y'_1-1|$ and $|y'_2|$.

\begin{lemma}[Control of the modulation parameters]\label{Lemma-param}
There exists $0<\alpha_2<\alpha_1$ such that if for all $t\geq 0$, $u(t)\in U_{\alpha_2}$, then $y_1(t)$ and $y_2(t)$ are $C^1$ functions of $t$ and they satisfy the following equations
\begin{eqnarray*}
(y'_1-1)[\|Q_{x_1}\|_2^2-(\ep, Q_{x_1x_1})]-y'_2(\ep, Q_{x_1x_2})=(L(Q_{x_1x_1}),\ep)+(R(\ep), Q_{x_1})
\end{eqnarray*}
and
\begin{eqnarray*}
-(y'_1-1)(\ep, Q_{x_1x_2})+y'_2[\|Q_{x_2}\|_2^2-(\ep, Q_{x_2x_2})]=(L(Q_{x_1x_2}),\ep)+(R(\ep), Q_{x_2}).
\end{eqnarray*}

Moreover, there exists a universal constant $C_2>0$ such that if $\|\ep(t)\|_2\leq \alpha$, for all $t\geq 0$, where $\alpha<\alpha_2$, then
\begin{equation}\label{ControlParam}
|y'_1-1|+|y'_2|\leq C_2 \|\ep(t)\|_2.
\end{equation}
\end{lemma}
\begin{proof}
Multiplying the equation \eqref{ep-eq2} by $Q_{x_1}$ and then by $Q_{x_2}$ and integrating by parts, we deduce
$$
\int \ep_t Q_{x_1}+\int L(\ep)Q_{x_1x_1}=(y'_1(t)-1)\left[\int Q^2_{x_1}-\ep Q_{x_1x_1}\right]+y'_2(t)\left[\int Q_{x_1}Q_{x_2}-\ep Q_{x_1x_2}\right]-\int R(\ep)Q_{x_1}
$$
and
$$
\int \ep_t Q_{x_2}+\int L(\ep)Q_{x_1x_2}=(y'_1(t)-1)\left[\int Q_{x_1}Q_{x_2}-\ep Q_{x_1x_2}\right]+y'_2(t)\left[\int Q^2_{x_2}-\ep Q_{x_2x_2}\right]-\int R(\ep)Q_{x_2}.
$$
We first note that the first term on the left hand sides of the last two equalities is zero in view of the orthogonality \eqref{ep-perp2}. Moreover, we also have $Q_{x_1}\perp Q_{x_2}$, thus, vanishing a couple of terms on the right hand sides. Therefore, using that $L$ is a self-adjoint operator, we obtain the following system
\begin{eqnarray*}
(y'_1-1)[\|Q_{x_1}\|_2^2-(\ep, Q_{x_1x_1})]-y'_2(\ep, Q_{x_1x_2})=(L(Q_{x_1x_1}),\ep)+(R(\ep), Q_{x_1})
\end{eqnarray*}
and
\begin{eqnarray*}
-(y'_1-1)(\ep, Q_{x_1x_2})+y'_2[\|Q_{x_2}\|_2^2-(\ep, Q_{x_2x_2})]=(L(Q_{x_1x_2}),\ep)+(R(\ep), Q_{x_2}).
\end{eqnarray*}
Next, we observe that there exists $\alpha_2>0$ such that
\begin{align*}
\begin{vmatrix}
\|Q_{x_1}\|_2^2-(\ep, Q_{x_1x_1}) & \quad -(\ep, Q_{x_1x_2}) \\
-(\ep, Q_{x_1x_2}) & \quad \|Q_{x_2}\|_2^2-(\ep, Q_{x_2x_2})
\end{vmatrix}
\geq \frac12 \|Q_{x_1}\|_2^2\|Q_{x_2}\|_2^2,
\end{align*}
if $\|\ep(t)\|_{2}\leq \alpha$ for every $t\geq 0$, where $\alpha<\alpha_2$.
Finally, we can find $C_2>0$ such that
\begin{align*}
|y'_1(t)-1|\leq &\left[ |(L(Q_{x_1x_1}),\ep)+(R(\ep), Q_{x_1})|\cdot|\|Q_{x_2}\|_2^2-(\ep, Q_{x_2x_2})|\right.\\
&\left.+|(L(Q_{x_1x_2}),\ep)+(R(\ep), Q_{x_2})|\cdot|(\ep, Q_{x_1x_2})|\right]\left[\frac12 \|Q_{x_1}\|_2^2\|Q_{x_2}\|_2^2\right]^{-1}\\
\leq & C_2 \|\ep(t)\|_{2}.
\end{align*}
Similarly, we obtain $|y'_2(t)|\leq C_2 \|\ep(t)\|_{2}$, concluding the proof.
\end{proof}

\section{Virial-type estimates and monotonicity}\label{S-5}
Our next step is to produce a virial-type functional which will help us study the stability properties of the solutions close to $Q$.
We first define a quantity depending on the $\ep$ variable, which incorporates the scaling generator $\Lambda$ and the eigenfunction of $L$ for the negative eigenvalue. For the gKdV version, compare with Combet \cite[\S 2.3]{Co10}. This turns out to play an important role in our instability proof, and what is absolutely crucial here is that we can find $\beta \neq 0$. We note that this does not work in the critical case ($p=3$), nor in the critical gKdV equation (with $p=5$ nonlinearity), since $\beta$ becomes zero. We now let
\begin{equation*}
F(x_1,x_2)=\int_{-\infty}^{x_1}(\Lambda Q (z, x_2)+\beta \chi_0(z,x_2)) \, dz,
\end{equation*}
where $\beta\in \R$ is a constant to be chosen later.

From the properties of $Q$ (see \eqref{prop-Q}) and $\chi_0$ (see \eqref{chi_0}), there exist $c, \delta>0$ such that
\begin{equation}\label{F-decay}
|F(x_1,x_2)|\leq c e^{-\frac{\delta}{2}|x_2|}\int_{-\infty}^{x_1}e^{-\frac{\delta}{2}|z|}dz,
\end{equation}
and therefore, $|F(x_1,x_2)|=o(e^{\frac{\delta}{2}x_1})$ when $x_1\rightarrow -\infty$ for every $x_2\in \R$ fixed. We also note that $F$ is a bounded function in $\R^2$, that is, $F\in L^{\infty}(\R^2)$.


We next define the virial-type functional
\begin{equation}\label{def-JA}
J(t)=\int_{\R^2}\ep(t)F(x_1,x_2)\, dx_1dx_2,
\end{equation}
and would like to show that it is well-defined. The first observation is that $J(t)$ is well-defined if $\ep(t)\in H^{3+}(\R^2)$ for all $t\geq0$. Indeed, this is a consequence of \cite[Theorem 2.3]{DeB}, and following the argument on pages 103-104 of \cite{DeB}, we find a universal constant $C_3>0$ independent of $t$ such that
\begin{equation}\label{Bound-J}
|J(t)|\leq C_3(t^{-3/4}+t^{1/2}).
\end{equation}
We remark that while our functional $J$ differs from the one used in \cite{DeB}, the estimate above is the same, see (4.2) in \cite{DeB}.

Our next task is to show that $J(t)$ is well-defined only assuming $\ep(t)\in H^{1}(\R^2)$ for all $t\geq0$. To this end, we adapt some monotonicity ideas introduced by Martel-Merle \cite{MM-KdV5} for the gKdV equation. Define
$$
\psi(x_1)=\frac{2}{\pi}\arctan{(e^{\frac{x_1}{M}})},
$$
where $M\geq 4$.
The following properties hold for $\psi$
\begin{enumerate}
\item
$\displaystyle \psi(0)=\frac{1}{2}$,\\

\item
$\displaystyle \lim_{x_1\rightarrow -\infty} \psi(x_1)=0$ and $\displaystyle \lim_{x_1\rightarrow +\infty} \psi(x_1)=1$,\\

\item
$1-\psi(x_1)=\psi(-x_1)$,\\

\item
$\displaystyle \psi^{\prime}(x_1)=\left(\pi M \cosh\left(\frac{x_1}{M}\right)\right)^{-1}$,\\

\item
$\displaystyle \left| \psi^{\prime\prime\prime}(x_1)\right|\leq \frac{1}{M^2}\psi^{\prime}(x_1)\leq \frac{1}{16}\psi^{\prime}(x_1)$.\\
\end{enumerate}

Let $(y_1(t), y_2(t))\in C^1(\R, \R^2)$ and for $x_0, t_0>0$ and $t\in [0,t_0]$ define
\begin{equation}\label{E:I}
I_{x_0,t_0}(t)=\int u^2(t, x_1, x_2)\psi(x_1-y_1(t_0)+\frac{1}{2}(t_0-t)-x_0)dx_1dx_2,
\end{equation}
where $u\in C(\R, H^1(\R^2))$ is a solution of the gZK equation \eqref{gZK}, satisfying
\begin{equation}\label{u-Q}
\|u(t, x_1+y_1(t), x_2+y_2(t))-Q(x_1,x_2)\|_{H^1}\leq \alpha,
\end{equation}
for some $\alpha>0$. While this is a similar concept to study the decay of the mass of the solution to the right of the soliton, as it was done in the gKdV equation in works of Martel-Merle, or for example, see our review of the instability of the critical gKdV case via monotonicity \cite{FHR1}, we note that the integration is two dimensional in the definition \eqref{E:I} (but the function $\psi$ is defined only in one variable $x_1$). We next study the behavior of $I$ in time and our almost monotonicity result is the following.
\begin{lemma}[Almost Monotonicity]\label{AM}
Let $M\geq 4$ fixed and assume that $y_1(t)$ is an increasing function satisfying $y_1(t_0)-y_1(t)\geq \frac{3}{4}(t_0-t)$ for every $t_0, t\geq 0$ with $t\in [0,t_0]$. Then there exist $\alpha_0>0$ and $\theta=\theta(M,p)>0$ such that if $u\in C(\R, H^1(\R^2))$ verify \eqref{u-Q} with $\alpha<\alpha_0$, then for all $x_0>0$, $t_0, t\geq 0$ with $t\in [0,t_0]$, we have
$$
I_{x_0,t_0}(t_0)-I_{x_0,t_0}(t)\leq \theta e^{-\frac{x_0}{M}}.
$$
\end{lemma}
\begin{proof}
Using the equation and the fact that $ \left| \psi^{\prime\prime\prime}(x)\right|\leq \frac{1}{M^2}\psi^{\prime}(x)\leq \frac{1}{16}\psi^{\prime}(x)$, we deduce
\begin{align}\label{Ix0t0}
\nonumber
\frac{d}{dt} I_{x_0,t_0}(t) = &2 \int uu_t \psi - \frac{1}{2} \int u^2\psi^{\prime}\\
\nonumber
= & -\int \left(3u_{x_1}^2+u_{x_2}^2-\frac{2p}{p+1}u^{p+1}\right)\psi^{\prime}+\int u^2\psi^{\prime\prime\prime}-\frac{1}{2}\int u^2\psi^{\prime}\\
\leq & -\int \left(3u_{x_1}^2+u_{x_2}^2+\frac{1}{4}u^2\right)\psi^{\prime}+\frac{2p}{p+1}\int u^{p+1}\psi^{\prime}.
\end{align}

Now, we estimate the last term on the right hand side of the previous inequality. First, we write
\begin{align}\label{RHS}
\int u^{p+1}\psi^{\prime}=\int Q(\cdot-\vec{y}(t))u^p\psi^{\prime}+\int (u-Q(\cdot-\vec{y}(t)))u^p\psi^{\prime},
\end{align}
where $\vec{y}(t)=(y_1(t), y_2(t))$. To estimate the second term, we use the Sobolev embedding $H^1(\R^2)\hookrightarrow L^{q}(\R^2)$, for all $2\leq q<+\infty$, to get
\begin{align}\label{RHS1}
\nonumber
\int (u-Q(\cdot-\vec{y}(t)))u^p\psi^{\prime}\leq& \|(u-Q(\cdot-\vec{y}(t)))u^{p-2}\|_{4/3}\|u^2\psi^{\prime}\|_{4}\\
\nonumber
\leq & c \, \|u-Q(\cdot-\vec{y}(t))\|_{2}\|u\|^{p-2}_{4(p-2)}\|u\sqrt{\psi^{\prime}}\|^2_{8}\\
\leq & c \, \alpha \, \|Q\|^{p-2}_{H^1} \int (|\nabla u|^2+|u|^2)\psi^{\prime}.
\end{align}
For the first term on the right hand side of \eqref{RHS}, we divide the integration into two regions $|\vec{x}-\vec{y}(t)|\geq R_0$ and $|\vec{x}-\vec{y}(t)|< R_0$, where $R_0$ is a positive number to be chosen later. Therefore, since $|Q(\vec{x})|\leq c \, e^{-\delta|\vec{x}|}$, we obtain
\begin{align*}
\nonumber
\int_{|\vec{x}-\vec{y}(t)|\geq R_0} Q(\cdot-\vec{y}(t))u^p\psi^{\prime}\leq& ce^{-\delta R_0}\|u^{p-2}\|_{3}\|u\sqrt{\psi^{\prime}}\|^2_{3}\\
\leq & c e^{-\delta R_0}\|Q\|^{p-2}_{H^1} \int (|\nabla u|^2+|u|^2)\psi^{\prime}.
\end{align*}
Next, when $|\vec{x}-\vec{y}(t)|\leq R_0$, we have that
\begin{align*}
\nonumber
\left|x_1-y_1(t_0)+\frac{1}{2}(t_0-t)-x_0\right| \geq& (y_1(t_0)-y_1(t)+x_0)-\frac{1}{2}(t_0-t)-|x_1-y_1(t)|\\
\geq & \frac{1}{4}(t_0-t)+x_0-R_0,
\end{align*}
where in the first inequality we have used that $y_1(t)$ is increasing, $t_0\geq t$ and $x_0>0$ to compute the modulus of the first term, and in the second line we have used the assumption $y_1(t_0)-y_1(t)\geq \frac{3}{4}(t_0-t)$.

Since $\psi^{\prime}(z)\leq c \, e^{-\frac{|z|}{M}}$, we use again the Sobolev Embedding $H^1(\R^2)\hookrightarrow L^{q}(\R^2)$, for all $2\leq q<+\infty$, to deduce that
\begin{align}\label{RHS3}
\nonumber
\int_{|\vec{x}-\vec{y}(t)|\leq R_0} Q(\cdot-\vec{y}(t))u^{p}\psi^{\prime}\leq& c\|Q\|_{\infty}e^{\frac{R_0}{M}}e^{-\frac{\left(\frac{1}{4}(t_0-t)+x_0\right)}{M}}\|u\|^p_{H^1}\\
\leq & c\|Q\|_{\infty}\|Q\|^p_{H^1}e^{\frac{R_0}{M}}e^{-\frac{\left(\frac{1}{4}(t_0-t)+x_0\right)}{M}}.
\end{align}

Therefore, choosing $\alpha$ such that $c \, \alpha \, \|Q\|^{p-2}_{H^1}<\frac{p+1}{2p} \cdot \frac{1}{4}$ and $R_0$ such that $c \, e^{-\delta R_0}\|Q\|^{p-2}_{H^1}<\frac{p+1}{2p} \cdot \frac{1}{4}$, collecting \eqref{RHS1}-\eqref{RHS3}, we obtain
$$
\frac{2p}{p+1}\int u^{p+1}\psi^{\prime}\leq \frac{1}{8}\int (|\nabla u|^2+|u|^2)\psi^{\prime}+c\|Q\|_{\infty}\|Q\|^p_{H^1}e^{\frac{R_0}{M}}e^{-\frac{\left(\frac{1}{4}(t_0-t)+x_0\right)}{M}}.
$$
Inserting the previous estimate in \eqref{Ix0t0}, there exists $C>0$ such that
\begin{align*}
\nonumber
\frac{d}{dt} I_{x_0,t_0}(t) \leq &-\int \left(\frac{3}{2}u_{x_1}^2+\frac{1}{2}u_{x_2}^2+\frac{1}{8}u^2\right)\psi^{\prime}+c \, e^{-\frac{x_0}{M}}\cdot e^{-\frac{1}{4M}(t_0-t)}\\
\nonumber
\leq & c \, e^{-\frac{x_0}{M}}\cdot e^{-\frac{1}{4M}(t_0-t)}
\end{align*}
Finally, integrating in time on $[t,t_0]$, we obtain the desired inequality for some $\theta=\theta(M,p)>0$.
\end{proof}

The next lemma is the main tool to obtain the upper bound for $|J(t)|$ independent of $t\geq 0$.
\begin{lemma}\label{AM2}
Let $y_1(t)$ satisfying the assumptions of Lemma \ref{AM}. Also assume that $y_1(t)\geq \frac{1}{2}t$ and $y_2(t)=0$ for all $t\geq 0$. Let $u\in C(\R, H^1(\R^2))$ a solution of the gZK equation \eqref{gZK} satisfying \eqref{u-Q} with $\alpha<\alpha_0$ (with $\alpha_0$ given in Lemma \ref{AM}) and with initial data $u_0$ verifying $\int |u_0({x_1,x_2})|^2 \, dx_2 \leq c \, e^{-\delta|{x_1}|}$ for some $c>0$ and $\delta>0$. Fix $M\geq \max\{4, \frac{2}{\delta}\}$. Then there exists $C=C(M,\delta, p)>0$ such that for all $t\geq 0$ and $x_0>0$
$$
\int_{\R}\int_{x_1>x_0}u^2(t, x_1+y_1(t), x_2) \, dx_1dx_2 \leq C \, e^{-\frac{x_0}{M}}.
$$
\end{lemma}
\begin{proof}
From Lemma \ref{AM} with $t=0$ and replacing $t_0$ by $t$, we deduce for all $t\geq 0$
$$
I_{x_0,t}(t)-I_{x_0,t}(0)\leq \theta e^{-\frac{x_0}{M}}.
$$
This is equivalent to
$$
\int u^2(t, x_1, x_2)\psi(x_1-y_1(t)-x_0)dx_1dx_2\leq \int u_0^2(x_1,x_2)\psi(x_1-y_1(t)+\frac{1}{2}t-x_0)dx_1dx_2+\theta e^{-\frac{x_0}{M}}.
$$
On the other hand,
\begin{align*}
\int u^2(t, x_1,x_2)\psi(x_1-y_1(t)-x_0)dx_1dx_2=&\int u^2(t, x_1+y_1(t), x_2)\psi(x_1-x_0)dx_1dx_2\\
 \geq & \frac{1}{2}\int_{\R}\int_{x_1>x_0}u^2(t, x_1+y_1(t), x_2)dx_1dx_2,
\end{align*}
where in the last inequality we have used the fact that $\psi$ is increasing and $\psi(0)=1/2$.

Now, since $-y_1(t)+\frac{1}{2}t\leq 0$ and $\psi$ is increasing, we have
$$
\int u_0^2(x_1,x_2)\psi(x_1-y_1(t)+\frac{1}{2}t-x_0)dx_1dx_2\leq \int u_0^2(x_1,x_2)\psi(x_1-x_0)dx_1dx_2.
$$
The assumptions $\int |u_0({x_1,x_2})|^2 \, dx_2 \leq c \, e^{-\delta|{x_1}|}$ and $\psi(x_1)\leq ce^{\frac{x_1}{M}}$ yield
\begin{align*}
\int u_0^2(x_1,x_2)\psi(x_1-x_0)dx_1dx_2
\leq & c \, \int e^{-\delta|{x_1}|} \, e^{\frac{x_1-x_0}{M}}dx_1\\
\leq & c \, e^{-\frac{x_0}{M}}\int e^{-\left(\delta-\frac{1}{M}\right)|{x_1}|}dx_1\\
\leq & \, \bar{C}(M,\delta) \, e^{-\frac{x_0}{M}},
\end{align*}
where in the last inequality we have used the fact that
$$
\delta -\frac{1}{M}\geq \frac{\delta}2 \Longleftrightarrow M \geq \frac{2}{\delta}.
$$
Collecting the above estimate, we obtain the desired inequality.
\end{proof}

Next, we compute the derivative of $J(t)$.
\begin{lemma}
Suppose that $\ep(t)\in H^1(\R^2)$, for all $t\geq 0$, and $y_1(t)$, $y_2(t)$ are two $C^1$ functions.  Then the function $t\mapsto J(t)$ is $C^1$ and
\begin{equation}\label{J'}
\frac{d}{dt}J=\beta \lambda_0 \int \ep\chi_0 +K(\ep),
\end{equation}
where, setting
\begin{equation}\label{E:beta}
\beta=-\frac{\int Q\Lambda Q}{\int Q\chi_0}
\end{equation}
and recalling the definition \eqref{R(ep)}, we have
\begin{equation}\label{K}
K(\ep)=\int Q\ep - (y_1'-1)\int \ep (\Lambda Q+\beta \chi_0)-y_2'\int \ep F_{x_2}-y_2'\int Q F_{x_2}-\int R(\ep)F.
\end{equation}
\end{lemma}
\begin{remark}
Since $Q$ and $\chi_0$ are both positive functions, we have $\int Q \chi_0>0$, and thus, $\beta$ in \eqref{E:beta} is well-defined.
By Lemma \ref{L-prop2}, part (iii), we deduce that $\ds \beta=-\frac{\int Q\Lambda Q}{\int Q\chi_0}>0$, if $p>3$. (Note that in the critical case, $p=3$, such a constant $\beta$ would be zero, since $\int Q \Lambda Q = 0$.)
\end{remark}
\begin{proof}
From the equation for $\ep$ \eqref{ep1-order1}, we have
\begin{align*}
\frac{d}{dt}J=& \int \ep_tF\\
=& \int (L \ep)_{x_1}F +(y'_1(t)-1)\int(Q+\ep)_{x_1}F + y'_2(t) \int (Q +\ep)_{x_2}F -\int  R(\ep)F\\
=&-\int \ep L(F_{x_1}) -(y'_1(t)-1)\int(Q+\ep)F_{x_1} - y'_2(t) \int (Q +\ep)F_{x_2} -\int  R(\ep)F\\
=&-\int \ep L(\Lambda Q+\beta \chi_0) -(y'_1(t)-1)\int(Q+\ep)(\Lambda Q+\beta \chi_0) - y'_2(t) \int (Q +\ep)F_{x_2} -\int  R(\ep)F.
\end{align*}
Using that $L(\Lambda Q)=-Q$ and $L(\chi_0)=-\lambda_0 \chi_0$ (see Theorem \ref{L-prop} and Lemma \ref{L-prop2}), we obtain
\begin{align*}
\frac{d}{dt}J
=&\beta \lambda_0\int \ep \chi_0 -(y'_1(t)-1)\int Q(\Lambda Q+\beta \chi_0) +\int \ep Q\\
&-(y'_1(t)-1)\int \ep (\Lambda Q+\beta \chi_0) - y'_2(t) \int (Q +\ep)F_{x_2} -\int  R(\ep)F.\\
\end{align*}
Setting $\beta=-\frac{\int Q\Lambda Q}{\int Q\chi_0}>0$, the second term in the last expression is zero, and then relation \eqref{J'} holds with $K(\ep)$ given by \eqref{K}.
\end{proof}

\section{$H^1$-instability of $Q$ for the supercritical gZK}\label{S-6}
In this section we prove Theorem \ref{Theo-Inst}. For $n\geq 1$ let
\begin{equation}\label{u0n}
u_{0,n}(\vec{x})=\lambda_nQ(\lambda_n \vec{x}), \quad \textrm{where} \quad \lambda_n=1+\frac1n \quad \textrm{and} \quad  \vec{x}=(x_1,x_2).
\end{equation}
The following proposition exhibits some properties of the sequence $\{u_{0,n}\}$ with respect to $Q$ that will be useful later.
\begin{proposition}\label{u0n-prop}
Let $u_{0,n}$ be given by \eqref{u0n}, then for every $n\in \mathbb{N}$
\begin{equation}\label{u0n-rel}
\|u_{0,n}\|_2=\|Q\|_2 \quad \textrm{and} \quad  E[u_{0,n}]<E[Q].
\end{equation}
Moreover,
\begin{equation}\label{u0n-lim}
\displaystyle \|u_{0,n}-Q\|_{H^1}\rightarrow 0, \quad \textrm{as} \quad n\rightarrow +\infty.
\end{equation}
\end{proposition}
\begin{proof}
A simple rescaling shows
$$
\int |u_{0,n}(\vec{x})|^2d\vec{x}=\int \lambda_n^2|Q(\lambda_n\vec{x})|^2d\vec{x}=\int |Q(\vec{x})|^2d\vec{x},
$$
which proves the equality in \eqref{u0n-rel}. Moreover, from $u_{0,n}(\vec{x})\rightarrow Q(\vec{x})$ for all $x\in \R^2$ and the exponential decay of $Q$ \eqref{prop-Q}, the limit in \eqref{u0n-lim} is true by dominated convergence theorem.

Next, we turn to the energy inequality in \eqref{u0n-rel}. Indeed,
\begin{align*}
E[u_{0,n}]=&\frac{1}{2}\int |\nabla u_{0,n}(\vec{x})|^2d\vec{x}-\frac{1}{p+1}\int |u_{0,n}(\vec{x})|^{p+1}d\vec{x}\\
=&\frac{1}{2}\int \lambda^4_n|\nabla Q(\lambda_n\vec{x})|^2d\vec{x}-\frac{1}{p+1}\int \lambda^{p+1}_n |Q(\lambda_n\vec{x})|^{p+1}d\vec{x}\\
=&\frac{\lambda^2_n}{2}\int |\nabla Q(\vec{x})|^2d\vec{x}-\frac{\lambda^{p-1}_n}{p+1}\int  |Q(\vec{x})|^{p+1}d\vec{x}.
\end{align*}
Therefore, using the Pohozaev identity
$$
\int |Q|^{p+1}=\frac{p+1}{p-1}\int |\nabla Q|^2
$$
and the definition of $\lambda_n=1+\frac{1}{n}$, we obtain
\begin{align*}
E[u_{0,n}]-E[Q]=&\left[\frac{p-1}{2}(\lambda_n^2-1)-(\lambda_n^{p-1}-1)\right]\frac{1}{p+1} \int |Q|^{p+1}d\vec{x}\\
=&\left[\frac{p-1}{2}\left(\frac{2}{n}+\frac{1}{n^2}\right)-\frac{p-1}{n} -\binom{p-1}{2}\frac{1}{n^2}-\sum_{k=3}^{p-1}\binom{p-1}{k}\frac{1}{n^k}\right] \frac{1}{p+1}\int |Q|^{p+1}d\vec{x}\\
=&\left[\left[\frac{p-1}{2}-\binom{p-1}{2}\right]\frac{1}{n^2}-\sum_{k=3}^{p-1}\binom{p-1}{k}\frac{1}{n^k}\right] \frac{1}{p+1}\int |Q|^{p+1}d\vec{x}.\\
\end{align*}
Since for $p>3$ we have
$$
\binom{p-1}{2}=\frac{(p-1)!}{(p-3)!2!}=\frac{(p-1)(p-2)}{2}>\frac{p-1}{2},
$$
we deduce the desired inequality.
\end{proof}

Now, assume by contradiction that $Q$ is stable. Then for every $\alpha>0$ there exists $n(\alpha)\in \mathbb{N}$ such that for every $t\geq 0$ we have
$$
u_{n(\alpha)}(t)\in U_{\alpha},
$$
where $u_{n(\alpha)}$ is the solution with initial data $u_{0,n(\alpha)}$. Since $u_{0,n(\alpha)}(\vec{x})=\lambda_{n(\alpha)}Q(\lambda_{n(\alpha)} \vec{x})$ is cylindrically symmetric, so will be $u_{n(\alpha)}(t)$ for all $t\geq 0$, as the equation is invariant under rotation in $\R_{x_1}$. Select $\alpha_0<\alpha_2<\alpha_1$, where $\alpha_1>0$ is given by Proposition \ref{ModThI} and $\alpha_2>0$ by Lemma \ref{Lemma-param}. To simplify the notation, we omit the index $n(\alpha_0)$ from now on. Definition \ref{eps} provides a function
\begin{equation}\label{ept}
\ep(t)=\ep_{({y_1}(t), y_2(t))}=u( x_1 +y_1(t), x_2+y_2(t)) - Q(x_1,x_2),
\end{equation}
satisfying \eqref{ep-perp2}, and, from the second conclusion in Proposition \ref{ModThI}, we also have $y_2(t)=0$ for all $t\geq 0$. From \eqref{u0n} and \eqref{u0n-lim}, we deduce that $y_1(0)=0$.

Since $u(t)\in U_{\alpha_0}$, from Proposition \ref{ModThI} and Lemma \ref{Lemma-param} we get
\begin{equation}\label{Bound_ep}
\|\ep(t)\|_{H^1}\leq C_1 \alpha_0 \quad \textrm{and} \quad |y'_1(t)-1|\leq C_2C_1 \alpha_0,
\end{equation}
so taking $\alpha_0<\{(2C_1)^{-1}, (4C_1C_2)^{-1}\}$, we obtain
\begin{equation*}
\|\ep(t)\|_{H^1}\leq 1 \quad \textrm{and} \quad \frac34\leq y'_1(t)\leq \frac54 \quad \textrm{for all} \quad t\geq 0.
\end{equation*}

The last inequality implies that $y_1(t)$ is increasing and by the Mean Value Theorem
$$
y_1(t_0)-y_1(t)\geq \frac{3}{4}(t_0-t),
$$
for every $t_0, t\geq 0$ with $t\in [0,t_0]$.
Also, recalling $y_1(0)=0$, another application of the Mean Value Theorem yields
$$
y_1(t)\geq \frac{1}{2}t,
$$
for all $t\geq 0$. Finally, using the two facts that $Q(\vec{x}) \leq c \, e^{-{\delta}|\vec{x}|}$ for some $c, {\delta} >0$, and $\lambda_n\geq 1$ for all $n\in \mathbb{N}$ by assumption \eqref{u0n}, we obtain
$$
|u_0(\vec{x})|\leq 2c\, e^{-{\delta}|\vec{x}|}.
$$

Now, from Lemma \ref{AM2}, we can deduce the following $L^2$ exponential decay on the right for $\ep(t)$.

\begin{corollary}\label{AM3}
Let $M\geq \max\{4, \frac{1}{\delta}\}$. If $\alpha_0>0$ is sufficiently small, there exists $C=C(M,\delta, p)>0$ such that for every $t\geq 0$ and $x_0>0$
$$
\int_{\R}\int_{x_1>x_0}\ep^2(t, x_1, x_2)dx_1dx_2\leq Ce^{-\frac{x_0}{2M}}.
$$
\end{corollary}
\begin{proof}
Applying Lemma \ref{AM2}, for a fixed $M\geq \max\{4, \frac{1}{\delta}\}$,
there exists $C=C(M,\delta,p)>0$ such that for all $t\geq 0$ and $x_0>0$, we have
\begin{equation}\label{u^2-decay}
\int_{\R}\int_{x_1>x_0}u^2(t, x_1+y_1(t), x_2)dx_1dx_2\leq Ce^{-\frac{x_0}{M}}.
\end{equation}
From the definition of $\ep(t)$ (see \eqref{ept}), recalling that $y_2(t)=0$ for all $t\geq 0$, we have that
\begin{equation*}
\ep\left(t, x_1, x_2\right)=u(t,x_1+y_1(t), x_2)-Q\left(x_1,x_2\right).
\end{equation*}
Moreover, since $Q(\vec{x}) \leq c \, e^{-{\delta}|\vec{x}|}$, we obtain
\begin{align}\label{Q^2-decay}
\nonumber
\int_{\R}\int_{x_1>x_0}Q^2\left(x_1,x_2\right)dx_1dx_2\leq &c\int_{\R}\int_{x_1>x_0}e^{-2{\delta}|\vec{x}|}dx_1dx_2\\
\nonumber
\leq & c\left(\int_{\R}e^{-{\delta}|x_2|}dx_2\right)\left(\int_{x_1>x_0}e^{-{\delta}x_1}dx_1\right)\\
\leq & \frac{c}{\delta^2}e^{-\delta{x_0}}\leq \frac{c}{\delta^2}e^{-\frac{x_0}{M}},
\end{align}
where in the last inequality we have used that $M\geq \frac{1}{\delta}$.
Finally, collecting \eqref{u^2-decay}-\eqref{Q^2-decay}, we deduce the desired result.
\end{proof}

In the next proposition, we obtain an upper bound for $|J(t)|$ independent of $t\geq 0$ (improving the bound \eqref{Bound-J} previously obtained by de Bouard \cite{DeB}).
\begin{proposition}
If $\alpha_0>0$ is sufficiently small, then there exists a constant $M_0>0$ such that
\begin{equation}\label{Bound-J2}
\left|J(t)\right|\leq M_0 \quad \textrm{for all} \quad t\geq 0.
\end{equation}
\end{proposition}
\begin{proof}
From the definition of $J(t)$ (see \eqref{def-JA}), we have that for all $t\geq 0$
\begin{align*}
|J(t)|\leq &\int_{\R}\int_{x_1\leq 0}\left|\ep(t,x_1,x_2)F(x_1,x_2)\right| dx_1dx_2+ \int_{\R}\int_{x_1>0}\left|\ep(t,x_1,x_2)F(x_1,x_2)\right| dx_1dx_2.
\end{align*}
Note that $F\chi_{\{x_1<0\}}\in L^2(\R^2)$ by \eqref{F-decay}, therefore, the first integral on the right hand side of the last inequality can be bounded as
$$
\int_{\R}\int_{x_1\leq 0}\left|\ep(t,x_1,x_2)F(x_1,x_2)\right| dx_1dx_2 \leq \|\ep(t)\|_2\|F\chi_{\{x_1<0\}}\|_2\leq c \, C_1\alpha_0,
$$
where in the last inequality we used \eqref{Bound_ep}.

On the other hand, again from \eqref{F-decay}, for every $x_2\in \R$ we have
$$
\sup_{x_1\in \R}|F(x_1,x_2)|\leq \frac{4c}{\delta} e^{-\frac{\delta}{2}|x_2|}.
$$
Therefore,
\begin{align*}
\int_{\R}\int_{x_1>0}\left|\ep(t,x_1,x_2)F(x_1,x_2)\right| dx_1dx_2
\leq & \int_{\R}\sup_{x_1\in \R}|F(x_1,x_2)|\left(\int_{x_1>0}\left|\ep(t,x_1,x_2)\right|dx_1\right)dx_2\\
\leq & \frac{4c}{\delta} \int_{\R} e^{-\frac{\delta}{2}|x_2|}\left(\int_{x_1>0}\left|\ep(t,x_1,x_2)\right|dx_1\right)dx_2.
\end{align*}
Furthermore,  using the Cauchy-Schwarz inequality in the $x_1$-variable, we get
\begin{align*}
\int_{x_1>0}\left|\ep(t,x_1,x_2)\right| dx_1= & \sum_{k=0}^{+\infty}\int_{k}^{k+1}\left|\ep(t,x_1,x_2)\right| dx_1\\
\leq & \sum_{k=0}^{+\infty}\left(\int_{k}^{k+1}\ep^2(t,x_1,x_2) dx_1\right)^{1/2}.
\end{align*}
Now, collecting the last two inequalities and using the Cauchy-Schwarz inequality in the $x_2$-variable, we obtain
\begin{align*}
\int_{\R}\int_{x_1>0}\left|\ep(t,x_1,x_2)F(x_1,x_2)\right| dx_1dx_2
\leq & \frac{4c}{\delta} \sum_{k=0}^{+\infty} \int_{\R} e^{-\frac{\delta}{2}|x_2|}\left(\int_{k}^{k+1}\ep^2(t,x_1,x_2) dx_1\right)^{1/2}dx_2\\
\leq & \frac{4c}{\delta} \sum_{k=0}^{+\infty} \left(\int_{\R} e^{-{\delta}|x_2|}dx_2\right)^{1/2}\left(\int_{\R}\int_{k}^{+\infty}\ep^2(t,x_1,x_2) dx_1dx_2\right)^{1/2}\\
\leq & \frac{4\sqrt{2}c}{\delta^{3/2}} C^{1/2}\sum_{k=0}^{+\infty} e^{-\frac{k}{4M}},
\end{align*}
assuming $\alpha_0>0$ is sufficiently small so that we can apply Corollary \ref{AM3} in the last inequality.
To complete the proof we take $M_0=cC_1\alpha_0+\frac{4\sqrt{2}c}{\delta^{3/2}} C^{1/2}\sum_{k=0}^{+\infty} e^{-\frac{k}{4M}}<+\infty$.
\end{proof}

The next theorem provides a strictly positive lower bound for $\ds \left|\frac{d}{dt}J(t)\right|$.
\begin{theorem}\label{Theo-J'}
If $\alpha_0>0$ is sufficiently small, then there exists a constant $a_0>0$ such that
\begin{equation*}
\left|\frac{d}{dt}J(t)\right|\geq a_0>0 \quad \textrm{for all} \quad t\geq 0.
\end{equation*}
\end{theorem}

\begin{proof}
Let $\alpha_0<\min\{\alpha_1, \alpha_2\}$, so that we can apply Lemmas \ref{ModThI} and \ref{Lemma-param}. In view of \eqref{J'}, we have
\begin{equation}\label{J'2}
\frac{d}{dt}J=\beta \lambda_0 \int \ep\chi_0 +K(\ep),
\end{equation}
where, since $y_2(t)=0$ for all $t\geq 0$,
\begin{equation*}
K(\ep)=\int Q\ep - (y_1'-1)\int \ep (\Lambda Q+\beta \chi_0)-\int R(\ep)F
\end{equation*}
and $R(\ep)$ satisfy \eqref{error}. We estimate the terms in $K(\ep)$ separately. First, observe that $\|\ep(t)\|\leq 1$ by \eqref{ep-control2}, if $\alpha_0<(2C_1)^{-1}$. Hence, from the Gagliardo-Nirenberg inequality \eqref{GN}, we deduce
\begin{align*}
\|R(\ep)\|_1\leq &C_0\sum_{k=2}^p(\|\ep\|_k^k+\|\ep_{x_1}\|_2\|\ep\|^{k-1}_{2(k-1)})\\
\leq & c\, C_0\sum_{k=2}^p(\|\nabla \ep\|_2^{k-2}\|\ep\|_2^2+\|\ep_{x_1}\|_2\|\nabla \ep\|_2^{k-2}\|\ep\|_2)\\
\leq &C_5 \| \ep\|_2\|\ep\|_{H^1},
\end{align*}
where $C_5=2pcC_0$.
Thus,
\begin{equation}\label{K1}
\left|\int R(\ep)F\right|\leq C_5\|F\|_{\infty}\| \ep\|_2\|\ep\|_{H^1}.
\end{equation}
On the other hand, from \eqref{ControlParam} and Cauchy-Schwarz inequality
\begin{equation*}
\left|(y_1'-1)\int \ep (\Lambda Q+\beta \chi_0)\right|\leq C_2\|\ep\|_2\left|\int \ep (\Lambda Q+\beta \chi_0)\right|\leq C_2\|\Lambda Q+\beta \chi_0\|_{2}\|\ep\|_2\|\ep\|_{H^1}.
\end{equation*}
Finally, the mass conservation \eqref{MC} for the solution $u$, the definition of $\ep$ \eqref{eq-ep2} and the relation \eqref{u0n-rel} imply
$$
\int Q^2 = \int u_0 = \int u(t)= \int Q^2 +2\int \ep Q +\int \ep^2
$$
and so
\begin{equation}\label{K3}
\left|\int \ep Q \right|\leq \frac12 \|\ep\|^2_2.
\end{equation}
Collecting \eqref{K1}-\eqref{K3}, there exists a universal constant (depending only on $p$) $C_6>0$ such that
\begin{equation}\label{K4}
\left| K(\ep) \right|\leq C_6 \|\ep\|_2\|\ep\|_{H^1}.
\end{equation}

Now, let
$$
\theta(t)=\int \ep(t)\chi_0.
$$
From Lemma \ref{Lemma-ort2} we deduce
$$
\theta^2(t)\geq \frac{k_1}{k_2}\|\ep(t)\|^2_2-\frac{1}{k_2}(L\ep(t),\ep(t))
$$
In order to estimate $(L\ep(t), \ep(t))$, we invoke Lemma \ref{Weinstein-Functional} to get
$$
W[u_0]=W[u(t)]= W[Q+\ep(t)]=W[Q]+\frac12 (L\ep(t), \ep(t))+H[\ep(t)]
$$
and then from the definition of the Weinstein's Functional $W$, we get
\begin{align*}
(L\ep(t), \ep(t))=& 2(E[u_0]-E[Q])+(M[u_0]-M[Q])-2H[\ep(t)]\\
= & -\delta_0-2H[\ep(t)],
\end{align*}
where $\delta_0>0$ by Proposition \ref{u0n-prop}. Thus,
\begin{align}\label{theta}
\theta^2(t)\geq & \frac{k_1}{k_2}\|\ep(t)\|^2_2+\frac{1}{k_2}(\delta_0+2H[\ep(t)])\nonumber \\
\geq & \frac{k_1}{2k_2}\|\ep(t)\|^2_2+\frac{\delta_0}{k_2},
\end{align}
where in the last line we chose $\alpha_0<k_1(4C'C_1k_2)^{-1}$ in order to use inequalities \eqref{H} and \eqref{ep-control2}.

Therefore, $\theta^2(t)$ is a strictly positive number for all $t\in \R$, and hence, the sign of $\theta(t)$ remains the same during the evolution. Let's assume that $\theta(t)$ is positive, then \eqref{theta} implies
$$
\theta(t)\geq \sqrt{ \frac{k_1}{2k_2}\|\ep(t)\|^2_2+\frac{\delta_0}{k_2}}\geq c \left(\sqrt{\frac{k_1}{2k_2}}\|\ep(t)\|_2+ \sqrt{\frac{\delta_0}{k_2}}\right).
$$
Plugging the last inequality in \eqref{J'2}, we obtain
$$
\frac{d}{dt}J\geq c \beta\lambda_0\sqrt{\frac{k_1}{2k_2}}\|\ep(t)\|_2 +c \beta\lambda_0\sqrt{\frac{\delta_0}{k_2}} +K(\ep).
$$
Finally, from \eqref{K4} we can choose $\alpha_0>0$, sufficiently small, such that
$$
\frac{d}{dt}J\geq \frac{c \beta\lambda_0}{2}\sqrt{\frac{k_1}{2k_2}}\|\ep(t)\|_2 +c \beta\lambda_0\sqrt{\frac{\delta_0}{k_2}}\geq c \beta\lambda_0\sqrt{\frac{\delta_0}{k_2}}>0.
$$
If $a(t)$ is negative, then arguing as above, we can show that for $\alpha_0>0$, sufficiently small, there exists $a_0>0$ such that $\ds \frac{d}{dt}J(t)\leq-a_0<0$, concluding the proof of Theorem \ref{Theo-J'}.
\end{proof}

\begin{proof}[Proof of Theorem \ref{Theo-Inst}]
If $\ds \frac{d}{dt}J(t)\geq a_0>0$, then integrating in $t$ variable both sides, we get
$$
J(t)\geq a_0t + J(0) \quad \textrm{for all} \quad t\geq 0,
$$
which is a contradiction with \eqref{Bound-J2}, the boundedness of $J(t)$ from above. The case when $J'(t)\leq -a_0<0$ goes absolutely similar.
\end{proof}

\bibliographystyle{amsplain}

\end{document}